\documentclass[11pt,reqno]{amsart}

\numberwithin{equation}{section}

\usepackage{amssymb}
\usepackage{amsmath}
\usepackage{amsbsy}
\usepackage{amscd}
\usepackage{amsthm}
\usepackage{amsfonts}
\usepackage{enumerate}

\usepackage{color}

\usepackage{ifpdf}
\ifpdf \usepackage[pdftex,pdfstartview=FitH,pdfpagemode=none,colorlinks,bookmarks,linkcolor=blue]{hyperref} \else  \usepackage[hypertex]{hyperref} \fi

\usepackage[nobysame,abbrev,alphabetic]{amsrefs}

\newcommand{\hide}[1]{}

\newtheorem{theorem}{Theorem}[section]

\newtheorem{lemma}[theorem]{Lemma}
\newtheorem{corollary}[theorem]{Corollary}

\newtheorem{conjecture}[theorem]{Conjecture}

\newtheorem{proposition}[theorem]{Proposition}

\newtheorem{remark}[theorem]{Remark}
\theoremstyle{definition}

\newcommand{\bC}{\mathbb{C}}

\newcommand{\bR}{\mathbb{R}}
\newcommand{\bZ}{\mathbb{Z}}

\newcommand{\bN}{\mathbb{N}}

\newcommand{\bT}{\mathbb{T}}

\newcommand{\talpha}{{\tilde{\alpha}}}

\newcommand{\bfone}{\mathbf{1}}

\newcommand{\di}{\mathrm{d}}

\newcommand{\Exp}{\mathop{\mathbb{E}}}

\newcommand{\onto}{\xymatrix{\ar@{>>}[r]&}}
\newcommand{\da}[4]{\xymatrix{#1 \ar@<.5ex>[r]^{#2} \ar@<-.5ex>[r]_{#3} & #4}}
\newcounter{subconst}[subsection]

\newcounter{const}

\newcounter{CONST}

\begin{document}

\title[M\"{o}bius function and discrete spectrum]{M\"{o}bius disjointness for topological models of ergodic systems with discrete spectrum}

\author[W. Huang]{Wen Huang}
\address{Department of Mathematics, Sichuan University,
Chengdu, Sichuan 610064, China}
\address{School of Mathematical Sciences, University of Science and Technology of
China, Hefei, Anhui 230026, China}
\email{wenh@mail.ustc.edu.cn}

\author[Z. Wang]{Zhiren Wang}
\address{Department of Mathematics, Pennsylvania State University,
University Park, PA 16802, USA}
\email{zhirenw@psu.edu}

\author[G. Zhang]{Guohua Zhang}
\address{School of Mathematical Sciences and Laboratory of Mathematics for Nonlinear Science, Fudan University,
Shanghai 200433, China}
\email{chiaths.zhang@gmail.com}

\setcounter{page}{1}

\begin{abstract}
We provide a criterion for a point satisfying the required disjointness condition in Sarnak's M\"obius Disjointness Conjecture. As a direct application, we have that the conjecture holds for any topological model of an ergodic system with discrete spectrum.
\end{abstract}

\maketitle

{\small\tableofcontents}

\section{Introduction}

The \emph{M\"obius function} $\mu:\bN\to\{-1,0,1\}$ is defined as follows: $\mu(n)=(-1)^k$ when $n$ is the product of $k$ distinct primes and $\mu(n)=0$ otherwise. The well-known \emph{M\"obius Randomness Law} in \cite[Section 13.1]{IK04} speculates that summing the M\"obius sequence $\mu(n)$ against any reasonable sequence $\xi(n)$ would lead to significant cancellations. This was verified in \cite{GT12} for polynomial nilsequences, a class of sequences of low complexity.

In \cite{S09} Sarnak reformulated  the law as the M\"obius Disjointness Conjecture by making precise the notion of a ``reasonable sequence'', namely a bounded sequence arising from a topological dynamical system with zero topological entropy. Since then, the M\"obius Disjointness Conjecture became an important theme in dynamical system and number theory.

Recall that a \emph{topological dynamical system} (\emph{TDS} for short) is a pair $(X, T)$ consisting of a compact metric space $X$, and a continuous self-map $T: X\rightarrow X$. The distance on $X$ will be denoted by $d(\cdot,\cdot)$.
In this paper, $\Exp$ stands for a finite average, for instance, $\displaystyle\Exp_{n=N_1}^{N_2-1} A_n= \frac1{N_2-N_1}\sum_{n=N_1}^{N_2-1}A_n.$

\begin{conjecture} \label{ConjSarnak} (M\"obius Disjointness Conjecture, \cite{S09})
 Let $(X, T)$ be a TDS with zero topological entropy. Then, for every $x\in X$,
\begin{equation}\label{EqConj}\lim_{N\to\infty}\Exp_{n=1}^{N}f(T^nx)\mu(n)=0, \forall f\in C(X).\end{equation}
\end{conjecture}

The case where $(X, T)$ is a finite periodic system is equivalent to the Prime Number Theorem in
arithmetic progressions, and the case where $T$ is a rotation on the circle is
Davenport's theorem \cite{D37}.
Many other special cases have been established for Conjecture \ref{ConjSarnak} more recently. Here we just list a few of them: \cite{MR10, G12, GT12, S12, B13b, B13a, BSZ13, ELD14,
MMR14, LS15,  KL15, MR15, FJ15, FKLM15, HLSY15, P15, W15, EKSL16, EKLR15, ELD15, Veech16, W16}.

Note that the conjecture holds for any minimal rotation over a compact abelian metric group, and hence for its every isomorphic extension (for details see \cite[Theorem 4.1]{DK15}, \cite[Proposition 5.2]{EKSL16} or \cite{Veech16}).
Thus it is natural to ask \emph{whether the conjecture holds for any topological model of an ergodic system with discrete spectrum}, as the minimal rotation over a compact abelian metric group is
the ``standard model" of an ergodic system with discrete spectrum by the well-known Halmos-von Neumann Representation Theorem (see for example \cite{HVN42} or \cite[Theorem 3.6]{W82}).

By a \emph{topological model} of an ergodic system $(\mathcal{X}, \mathcal{B}, \xi, S)$ we mean any uniquely ergodic TDS $(X, T)$ (with the unique invariant Borel probability measure $\nu$) such that $(X, \mathcal{B}_X, \nu, T)$ and $(\mathcal{X}, \mathcal{B}, \xi, S)$ are measure-theoretically isomorphic, where $\mathcal{B}_X$  is the Borel $\sigma$-algebra of $X$.
Recall that a TDS is \emph{uniquely ergodic} if it admits a uniquely invariant Borel probability measure.
We say that TDS $(X, T)$ is \emph{minimal} if $X$ is the only nonempty closed \emph{invariant} subset $K$ (that is, $T K\subset K$) of the system.

Though an ergodic system with discrete spectrum may be the simplest ergodic system with zero measure-theoretic entropy, the question seems to be not easy, even we require additionally that these topological models are minimal. The trivial example of an ergodic system with discrete spectrum is a finite periodic system, whose each minimal topological model is exactly a finite periodic system, and so
the conjecture follows from the Prime Number Theorem in arithmetic progressions. In general, an ergodic system with discrete spectrum need not to be a finite periodic system, whose minimal topological model may present very complicated dynamical behavior, except its standard model of a minimal rotation over a compact abelian metric group. In fact, Lehrer \cite{L87} showed that each non-periodic ergodic system admits a minimal topological model which is topologically strongly mixing.

In this paper we will solve the question by proving the following result (without the assumption of the topological model being a minimal TDS).

\begin{theorem} \label{discrete-spectrum}
Conjecture \ref{ConjSarnak} holds for all topological models of an ergodic system with discrete spectrum.
\end{theorem}

Remark that the conjecture was solved for any topological model of an ergodic system with \emph{irrational} discrete spectrum by
 the recent interesting work of El Abdalaoui, Lema\'ncyzk and de la Rue \cite{ELD15}\footnote
{In fact in \cite{ELD15} Conjecture \ref{ConjSarnak} is solved for any topological model of an ergodic system with \emph{quasi-discrete spectrum}. A transformation with quasi-discrete spectrum is by definition totally ergodic \cite{ELD15}*{Definition 2}, which holds if and only if all of its eigenvalues except 1 are irrational. In particular, any ergodic automorphism with irrational discrete spectrum has quasi-discrete spectrum.}. In \cite{DG15}, Downarowicz and Glasner asked if the conjecture holds for all topological models of an ergodic system with \emph{rational} discrete spectrum $\{- 1, 1\}$ (whose standard model is the two-point periodic system). Thus Theorem \ref{discrete-spectrum} answers affirmatively the question raised in \cite{DG15}. After finishing a preprint version of the paper, we learned from Mariusz Lema\'nczyk that El Abdalauoi, Ku\l{}aga-Przymus, Lema\'nczyk and de la Rue had previously solved the conjecture for all topological models of a finite periodic system and communicated it to others on several occasions.

Theorem \ref{discrete-spectrum} follows directly from the following more general result.

\begin{theorem}\label{ThmMain}
Suppose that TDS $(X,T)$ admits only countably many ergodic invariant Borel probability measures, and that each of these measures has discrete spectrum. Then Conjecture \ref{ConjSarnak} holds for the system $(X, T)$.
\end{theorem}

We shall prove Theorem \ref{ThmMain} by showing the following stronger result.

\begin{theorem}\label{ThmGeneral}
Let $(X,T)$ be a TDS with $x\in X$ and $\{N_i: i\in \mathbb{N}\}\subset \mathbb{N}$ tend to $\infty$, such that the sequence $\displaystyle\Exp_{n=1}^{N_i}\delta_{T^n (x)}$ converges to a Borel probability measure $\rho$ (which is obviously invariant). Suppose that $\rho$ is a convex combination of countably many ergodic invariant Borel probability measures, and that each of these ergodic measures has discrete spectrum. Then
\begin{equation}\label{EqThmGeneral}
\lim_{i\to\infty}\Exp_{n=1}^{N_i}f(T^nx)\mu(n)=0, \forall f\in C(X).
\end{equation}
\end{theorem}

Our criterion Theorem \ref{ThmGeneral} provides a sufficient condition for a point satisfying the disjointness condition \eqref{EqConj}, even though it is not hard to see that the point may produce positive topological entropy.
Observe that, as shown in \cite{DK15} and \cite{EKSL16}, the equation \eqref{EqConj} may fail for some point from a class of Toeplitz sequences which may have positive topological entropy.

As shown in next section, Theorem \ref{ThmGeneral}, the main technical result of the paper, is used to deduce Theorem \ref{ThmMain} and recover some recent results obtained by other mathematicians. We now briefly outline its proof as follows.

Given the point $x$, the sequence $\{N_i\}$ and the limit measure $\rho$, we first choose finitely many distinct ergodic measures $\rho_1,\cdots,\rho_J$ such that up to $\epsilon$-error, $\rho$ is a convex combination of $\rho_1,\cdots, \rho_J$. Thanks to the Halmos-von Neumann Representation Theorem, one can approximate, in a measurable way, the dynamics on $\rho_j$ by a rotation of a compact abelian metric group. And then by Lusin Theorem, one can fix a compact set $A_j$ on which this approximation is continuous. As the $\rho_j$'s are mutually singular, one can make the $A_j$'s disjoint.
Then for a typical $1\leq n\leq N_i$, $T^nx$ is very close to one of these $A_j$'s. By continuity of the map $T$, this allows us to approximate the finite orbit segment $T^nx,\cdots, T^{n+L-1}x$ by a finite orbit segment that stays inside $A_j$ for most of the time. After further approximating measurable functions on the rotation by continuous ones on a finite dimensional torus, we may approximate the values of the function $f$ observed along the finite segment of orbit above by a linear combination of exponential sequences of the form $\{e^{2\pi i\beta (l+n)}\}_{l=0}^{L-1}$. The value of $\beta$ may depend on the starting position $n$ of the segment. We finally apply a recent theorem of Matom\"aki, Radziwi\l{}\l{} and Tao on short averages of non-pretentious multiplicative functions to assert that, for sufficiently large $N_i$ and appropriately chosen $L$, such exponential sequences have significant cancellations against the M\"obius sequence $\{\mu(l+n)\}_{l=0}^{L-1}$ for most $1\leq n\leq N_i$. Finally, one combines the bounds on these short segments to obtain the conclusion.

\vskip 10pt

\noindent {\bf Acknowledgements.}
Part of the work was carried out during a visit of Z. Wang to the School of Mathematical Sciences and Shanghai Key Laboratory for Contemporary Applied Mathematics of Fudan University. He gratefully acknowledges the hospitality of Fudan University.

We thank Ai-Hua Fan, Yunping Jiang, Peter Sarnak, Weixiao Shen and Xiangdong Ye for helpful and encouraging discussions.

We also thank Lema\'nczyk for bringing to our attention the work of \cite{EKSL16} and \cite{Veech16}, and for informing us, upon the posting of a preprint version of this paper, of El Abdalauoi, Ku\l{}aga-Przymus, Lema\'nczyk and de la Rue's earlier unpublished proof for topological models of a finite periodic system.

W. Huang was supported by NSFC (11225105 and 11431012),
Z. Wang was supported by NSF (DMS-1451247 and DMS-1501095), and G. Zhang was supported by NSFC (11271078 and 11671094).

\section{Consequences of Theorem \ref{ThmGeneral}}

In this section, we present and prove some consequences of Theorem \ref{ThmGeneral}.

\subsection{Direct consequences of Theorem \ref{ThmGeneral}}

Firstly we can deduce Theorem \ref{ThmMain} easily from Theorem \ref{ThmGeneral} as follows.

\begin{proof}[Proof of Theorem \ref{ThmMain}]
Let $(X, T)$ be as in the statement of the theorem. Given $x\in X$ and $f\in C(X)$, it suffices to prove that in any increasing sequence $\{\tilde N_i\}$ of positive integers, there is a subsequence $\{N_i\}$ for which the convergence \eqref{EqThmGeneral} holds.
Indeed, one can always choose a subsequence $\{N_i\}$ such that $\displaystyle\Exp_{n=1}^{N_i}\delta_{T^nx}$ converges to a Borel probability measure $\rho$. By the assumption, $\rho$ is a convex combination of countably many ergodic invariant Borel probability measures, and each of these ergodic measures has discrete spectrum. Thus \eqref{EqThmGeneral} holds by Theorem \ref{ThmGeneral}.
\end{proof}

Observe that if the compact metric state space of a TDS contains at most countably many points, then the system satisfies the assumption of Theorem \ref{ThmMain}, as in this case ergodic invariant Borel probability measures of the system are supported on disjoint periodic orbits. And so we recover the following recent result obtained by Wei \cite{W16} as a direct corollary of Theorem \ref{ThmMain}.

\begin{theorem}\cite{W16}*{Theorem 5.16} \label{countably}
Assume that $X$ contains at most countably many points. Then
Conjecture \ref{ConjSarnak} holds for the system $(X, T)$.
\end{theorem}

Recently, in \cite{FJ15} Fan and Jiang related the M\"obius disjointness with the notion of \emph{stable in the mean in the sense of Lyapunov} or simply \emph{mean-L-stable}, which was introduced in \cite{F51} by Fomin
 when studying ergodic systems with discrete spectrum and further discussed in \cite{O52, A59, LTY15, DG15, FJ15}.

In the following we will deduce from Theorem \ref{ThmGeneral} the following Proposition \ref{ProGeneral}, which allows us to recover the M\"obius function case \cite[Corollary 1]{FJ15} of Fan and Jiang's recent work \cite {FJ15}.
As implied by a characterization due to Downarowicz-Glasner \cite[Theorem 2.1]{DG15} and Li-Tu-Ye \cite[Theorem 3.8]{LTY15}, each minimal mean-L-stable TDS is an isomorphic extension of a minimal rotation over a compact abelian metric group, and then any minimal mean-L-stable TDS (and hence the system considered in Proposition \ref{ProGeneral}) has zero topological entropy. See \cite{DG15} for the detailed definition of an isomorphic extension.

Recall that a TDS $(X, T)$ is \emph{mean-L-stable} if for any $\epsilon> 0$ there is $\delta> 0$ such that $d (x_1, x_2)< \delta$ implies $d (T^n x_1, T^n x_2)< \epsilon$ for all $n\in \bN$ except a set of upper density less than $\epsilon$.
 A \emph{minimal} set is a nonempty closed invariant set $K\subset X$ such that the subsystem $(K, T)$ is minimal.
 Now let $\nu$ be an invariant Borel probability measure of $(X,T)$ and $K\in \mathcal{B}_X$. We say that $\nu$ is \emph{supported} on $K$ if $\nu (K)= 1$.

\begin{proposition}\label{ProGeneral}
Let $(X,T)$ be a TDS with $x\in X$ and $\{N_i: i\in \mathbb{N}\}\subset \mathbb{N}$ tend to $\infty$, such that the sequence $\displaystyle\Exp_{n=1}^{N_i}\delta_{T^n (x)}$ converges to a Borel probability measure $\rho$ (which is obviously invariant). Suppose that $\rho$ is supported on countably many minimal mean-L-stable subsystems $(X_j, T)$. Then
\begin{equation*}\label{EqProGeneral1}
\lim_{i\to\infty}\Exp_{n=1}^{N_i}f(T^nx)\mu(n)=0, \forall f\in C(X).
\end{equation*}
\end{proposition}
\begin{proof}
Each TDS $(X_j, T)$ is uniquely ergodic by \cite[(6.4)]{O52} (see also \cite[Corollary 3.4]{LTY15}), and set $\nu_j$ to be the unique ergodic corresponding invariant Borel probability measure. In particular,
$\rho$ is a convex combination of these countably many $\nu_j$. By the characterization of a minimal mean-L-stable TDS in \cite{DG15, LTY15}, each $(X_j, T)$ is an isomorphic extension of a minimal rotation over a compact abelian metric group, and hence $\nu_j$ has discrete spectrum. Then the conclusion follows from Theorem \ref{ThmGeneral}.
\end{proof}

As a byproduct of Proposition \ref{ProGeneral} we have the following result. Note that each mean-L-stable TDS has zero topological entropy (see Remark \ref{RmkMLSEnt}).

\begin{theorem}\label{ThmMLS}
Conjecture \ref{ConjSarnak} holds for any mean-L-stable TDS.
\end{theorem}
\begin{proof}
Let $(X, T)$ be a mean-L-stable TDS and $x\in X$. Set $X_*$ to be the closure of the orbit $\{T^n x: n\in \bN\}$. Then the TDS $(X_*, T)$ is uniquely ergodic (again by \cite[(6.4)]{O52} or \cite[Corollary 3.4]{LTY15}), and set $\nu_x$ to be the unique corresponding invariant Borel probability measure. In particular, the sequence $\displaystyle\Exp_{n= 1}^N \delta_{T^n (x)}$ converges to $\nu_x$, which is necessarily supported on the unique minimal subsystem of $(X_*, T)$, which is also mean-L-stable. Then the conclusion follows from Proposition \ref{ProGeneral}.
\end{proof}

\subsection{An alternative proof of Theorem \ref{ThmMLS}}

In fact, with the help of Proposition \ref{ProGeneral}, we can recover \cite[Corollary 1]{FJ15} and show that \cite[Corollary 1]{FJ15} also implies Theorem \ref{ThmMLS}.
Following \cite{FJ15}, we say that a TDS $(X, T)$ is \emph{minimally mean attractable} if for each $x\in X$ there exists a minimal subset $M_x$ of $(X, T)$ such that $x$ is \emph{mean attracted} to $M_x$, that is, for any $\epsilon>0$ there is $z\in K_x$ with
$\displaystyle\limsup_{N\rightarrow \infty} \Exp_{n= 1}^N d (T^n x, T^n z)< \epsilon$; and that $(X, T)$ is \emph{minimally mean-L-stable} if every minimal subsystem is mean-L-stable.

We remark that the second part of Proposition \ref{mmammls} is in fact \cite[Corollary 1]{FJ15}. We give here an alternative proof of it based on Proposition \ref{ProGeneral}.

\begin{proposition} \label{mmammls}
Let $(X, T)$ be a TDS.
\begin{itemize}

\item If $(X, T)$ is mean-L-stable, then it is both minimally mean-L-stable and minimally mean attractable.

\item
If $(X, T)$ is both minimally mean-L-stable and minimally mean attractable, then it satisfies Conjecture \ref{ConjSarnak}.
\end{itemize}
\end{proposition}
\begin{proof}
First assume that $(X, T)$ is mean-L-stable. It is obviously minimally mean-L-stable. Let $x\in X$, and set $X_*$ to be the closure of $\{T^n x: n\in \bN\}$. Then TDS $(X_*, T)$ is uniquely ergodic and hence contains a uniquely minimal subset $X_x$. Let $\epsilon> 0$. By \cite[Lemma 3.1]{LTY15} we can select $\delta> 0$ such that $d (x_1, x_2)< \delta$ implies $\displaystyle\limsup_{N\rightarrow \infty} \Exp_{n= 1}^N d (T^n x_1, T^n x_2)< \epsilon$. Take $m\in \bN$ with $d (T^m x, X_x)< \delta$ and choose $x_*\in X_x$ with $d (T^m x, T^m x_*)< \delta$. Thus $\displaystyle\limsup_{N\rightarrow \infty} \Exp_{n= 1}^N d (T^n x, T^n x_*)< \epsilon$, and so $(X, T)$ is minimally mean attractable.

Now assume that $(X, T)$ is not only minimally mean-L-stable but also minimally mean attractable, and let $x\in X$. Fix any $f\in C (X)$ and $\epsilon> 0$. Take $\delta> 0$ such that $d (x_1, x_2)< \delta$ implies $|f (x_1)- f (x_2)|< \frac{\epsilon}{2}$ and choose $M> 1$ to be a finite upper bound for $|f|$. By the assumption, $x$ is mean attracted to a minimal subset $M_x$ of $(X, T)$, and so there is $z\in M_x$ with
$\displaystyle\limsup_{N\rightarrow \infty} \Exp_{n= 1}^N d (T^n x, T^n z)< \frac{\delta \epsilon}{4 M}$,
which implies
\begin{equation} \label{estimate}
\frac{\epsilon}{4 M}> \limsup_{N\rightarrow \infty} \frac{1}{N} \# \{1\le n\le N: d (T^n x, T^n z)\ge \delta\}.
\end{equation}
By the construction of $\delta$, it is easy to obtain $\displaystyle\limsup_{N\rightarrow \infty} \Exp_{n= 1}^N |f (T^n x)- f (T^n z)|< \epsilon$ from \eqref{estimate}. Again by the assumption $(M_x, T)$ is mean-L-stable, and then it admits a unique invariant Borel probability measure $\nu_x$. In fact, the sequence $\displaystyle\Exp_{n= 1}^N \delta_{T^n (z)}$ converges to $\nu_x$, and hence $\displaystyle\limsup_{N\rightarrow \infty} \left|\Exp_{n= 1}^N f (T^n x)- \int_X f \di \nu_x\right|< \epsilon$. By the arbitrariness of $\epsilon$ and $f$ we obtain that the sequence $\displaystyle\Exp_{n= 1}^N \delta_{T^n (x)}$ also converges to the measure $\nu_x$.
Now applying Proposition \ref{ProGeneral} to the point $x$ we obtain the conclusion.
\end{proof}

\begin{remark}\label{RmkMLSEnt}
By the proof of Proposition \ref{mmammls}, if a TDS $(X, T)$ is both minimally mean-L-stable and minimally mean attractable, then each ergodic invariant Borel probability measure is supported on a minimal mean-L-stable subsystem and hence has zero measure-theoretic entropy by \cite{DG15, LTY15}, thus the system $(X, T)$ has zero topological entropy. In particular, by Proposition \ref{mmammls}, mean-L-stability implies zero topological entropy.
\end{remark}

\section{Proof of Theorem \ref{ThmGeneral}}

From now on, we shall fix the point $x\in X$, the sequence $\{N_i\}$, the measure $\rho$ as in Theorem \ref{ThmGeneral}. By the assumption, we can write $\rho$ as a countable average $\sum_{i} w_i \rho_i$, where $\sum_i w_i=1$ with each $w_i>0$, and the $\rho_i$'s are distinct ergodic invariant Borel probability measures of $(X, T)$ and each of these ergodic measures has discrete spectrum.

We shall also fix any $f\in C (X)$ and $\epsilon\in(0,\frac1{30})$. Without loss of generality, we assume $|f|\leq 1$. It suffices to prove
\begin{equation}\label{EqEpsilon}
\limsup_{i\to\infty}\left|\Exp_{n=1}^{N_i}f(T^nx)\mu(n)\right|\leq 30 \epsilon.
\end{equation}

\subsection{Approximation by functions on tori}

Clearly there are finitely many of $\rho_i$'s, say $\rho_1, \cdots, \rho_J$, such that
\begin{equation}\label{EqFiniteMin}\sum_{j=1}^Jw_j>1-\epsilon.\end{equation}

Denote $\bT^d=\bR^d/\bZ^d$ for $d\in \mathbb{N}$ and $e(\theta)=\exp(2\pi i\theta)$ for $\theta\in\bT^1$ or $\bR$.

\begin{proposition}\label{PropProjApprox}For some integer $d\geq 1$, there exist
\begin{itemize}
\item a continuous function $h\in C(X)$ with $|h|<2$ such that
\begin{equation}\label{EqPropProjApprox1}\int_X|f-h|\di\rho<7\epsilon;\end{equation}
\item mutually disjoint compact subsets $A_1,\cdots, A_J\subset X$ with each
\begin{equation}\label{EqPropProjApprox2}\rho_j(A_j)>1-\epsilon^2;\end{equation}
\item for each $1\leq j\leq J$, a vector $\alpha_j\in\bT^d$ and a continuous map $p_j: A_j\rightarrow \bT^d$ such that for each nonnegative integer $l$ and any point $x_*\in X$, if $x_*$ and $T^l x_*$ are both contained in $A_j$, then
    \begin{equation}\label{EqPropProjApprox3}p_j(T^l x_*)=R_{\alpha_j}^l (p_j(x_*)),\end{equation}
     where
$R_{\alpha_j}: \bT^d\rightarrow \bT^d$ stands for the rotation $z\mapsto z+ \alpha_j$;

\item for each $1\leq j\leq J$, a continuous function $h'_j\in C(\bT^d)$ of the form \begin{equation}\label{EqPropProjApprox4a}h'_j(z)=\sum_{m=1}^{M_j}a_{j,m}e(\xi_{j,m}\cdot z),\end{equation} with integer $M_j\ge 1$, coefficients $a_{j,1},\cdots, a_{j,M_j}\in\bC$ and frequencies $\xi_{j,1},\cdots, \xi_{j,M_j}\in\bZ^d$, such that $|h'_j|<2$ and
\begin{equation}\label{EqPropProjApprox4}h(x_*)=h'_j(p_j(x_*)), \forall x_*\in A_j.\end{equation}
\end{itemize}
 \end{proposition}
\begin{proof}
By the Halmos-von Neumann Representation Theorem, any ergodic system with discrete spetrum is measurably isomorphic to a minimal rotation over a compact abelian metric group (with the normalized Haar measure).
It is standard that  a compact abelian metric group is topologically isomorphic to a closed subgroup of $\mathbb{T}^{\mathbb{N}}$ (see for example
\cite[Chapter 5, Corollary 1]{M77}).

Hence for each $j=1,\cdots, J$, there exist $\talpha_j\in \mathbb{T}^{\mathbb{N}}$ and a $\widetilde{R}_{\talpha_j}$-invariant ergodic Borel probability measure $\widetilde{\theta}_j$ on $\mathbb{T}^\mathbb{N}$,
 such that $(X,\mathcal{B}_X,\rho_j,T)$ is measurably isomorphic to $(\mathbb{T}^{\mathbb{N}}, \mathcal{B}_{\mathbb{T}^{\mathbb{N}}},\widetilde{\theta}_j,\widetilde{R}_{\talpha_j})$,
where $\widetilde{R}_{\talpha_j}(\omega)=\omega+\talpha_j$ for each $\omega\in \mathbb{T}^{\mathbb{N}}$.
 We write $\phi_j: (X,\mathcal{B}_X,\rho_j,T)\rightarrow (\mathbb{T}^{\mathbb{N}}, \mathcal{B}_{\mathbb{T}^{\mathbb{N}}},\widetilde{\theta}_j,\widetilde{R}_{\talpha_j})$ for the measurable isomorphism. Then there exists a subset $Y_j$ of full $\rho_j$-measure such that $\phi_j: Y_j\rightarrow \phi_j(Y_j)$ is invertible and measure-preserving and
\begin{equation}\label{EqPropProjApprox3a}
T Y_j\subset Y_j\ \text{and}\ \phi_j (T x_*)=\widetilde{R}_{\talpha_j} (\phi_j(x_*)), \forall x_*\in Y_j.
\end{equation}

Note that these countably many $\rho_i$'s are mutually singular, it makes no any difference to assume that these $Y_1, \cdots, Y_J$ are mutually disjoint, and additionally that $\rho_i (Y_j)= 0$ for each $1\le j\le J$ and all $i\neq j$ (no matter whether $i\in \{1, \cdots, J\}$ or not).

For each $1\leq j\leq J$, by Lusin's Theorem there exists a compact subset $A_j\subset Y_j$ such that $\rho_j(A_j)>1- \epsilon^2$ and $\phi_j|_{A_j}$, the restriction of $\phi_j$ over $A_j$, is continuous. Thus $\phi_j(A_j)$ is a compact subset of $\bT^\bN$ and $\phi_j|_{A_j}: A_j\rightarrow \phi_j(A_j)$ is a homeomorphism, moreover, $\widetilde{\theta}_j (\phi_j(A_j))=\rho_j(A_j)>1- \epsilon^2$.

Now for each $1\leq j\leq J$, the function $f\circ(\phi_j|_{A_j})^{-1}$ is continuous over $\phi_j(A_j)$, and then by Tietze Extension Theorem there exists a function $g'_j\in C(\mathbb{T}^{\mathbb{N}})$ with $|g'_j|\le 1$ and
$g'_j|_{\phi_j(A_j)}=f\circ (\phi_j|_{A_j})^{-1}$. Thus there exist $d\in \mathbb{N}$ and functions $h'_1,\cdots, h'_J\in C(\mathbb{T}^d)$
such that
\begin{align}\label{ineq-2}
\max_{1\le j\le J} \max_{\mathbb{T}^{\mathbb{N}}} |g'_j- h'_j\circ \pi_{d}|<\epsilon,
\end{align}
where $\pi_d: \mathbb{T}^{\mathbb{N}}\rightarrow \mathbb{T}^d$ denotes the projection of $\mathbb{T}^{\mathbb{N}}$ to the first $d$ coordinates.
As the Fourier basis $\{e(\xi\cdot z):\xi\in \mathbb{Z}^d\}$ generates a dense subspace of $C(\mathbb{T}^d)$, one may assume without loss of generality that each
$h'_j$ has the form of
\begin{align*}
h'_j(z)=\sum_{m=1}^{M_j}a_{j,m}e(\xi_{j,m}\cdot z)
\end{align*}
for some integer $M_j\in \mathbb{N}$, coefficients $a_{j,1},\cdots, a_{j,M_j}\in\bC$ and frequencies $\xi_{j,1},\cdots, \xi_{j,M_j}\in\bZ^d$.
It is clear that $|h'_j|\le 1+ \epsilon< 2$.

Now define $h_j\in C(A_j)$ as $h_j=h'_j\circ\pi_d\circ\phi_j|_{A_j}$ for each $1\le j\le J$. As these compact subsets $A_j\subset Y_j$ are mutually disjoint, by Tietze Extension Theorem we can find a function $h\in C(X)$ with $|h|< 2$ and $h=h_j$ over each $A_j$. Note that by \eqref{ineq-2}, $|f-h|=|f-h_j|<\epsilon$ over each $A_j$.
Furthermore,
\begin{equation}\label{EqTotalMass}\begin{aligned}\rho\left(\bigsqcup_{j=1}^JA_j\right)
=&\sum_{j=1}^J \rho (A_j)= \sum_{j=1}^Jw_j\rho_j(A_j)\\
>&\sum_{j=1}^Jw_j(1-\epsilon^2)\\
> &(1-\epsilon)(1-\epsilon^2)\ (\text{using \eqref{EqFiniteMin}})> 1-2\epsilon.
\end{aligned}\end{equation}

For each $1\le j\le J$, we take $p_j=\pi_{d}\circ \phi_j$ and $\alpha_j=\pi_{d}(\talpha_j)$.
Note that
\begin{equation*}\begin{aligned}
\int_X|f-h|\di\rho=&\int_{\bigsqcup_{j=1}^JA_j}|f-h|\di\rho+\int_{X\setminus\bigsqcup_{j=1}^JA_j}|f-h|\di\rho\\
=&\sum_{j=1}^J\int_{A_j}|f-h_j|w_j\di\rho_j+\int_{X\setminus\bigsqcup_{j=1}^JA_j}|f-h|\di\rho\\
\leq &\sum_{j=1}^J\epsilon w_j\rho_j(A_j)+\left(1-\rho\Big(\bigsqcup_{j=1}^JA_j\Big)\right)\cdot\left(\max_X|f|+\max_X|h|\right)\\
<&\epsilon+2\epsilon\cdot3\ (\text{using \eqref{EqTotalMass}})\ =7\epsilon,
\end{aligned}\end{equation*}
which implies
the inequality \eqref{EqPropProjApprox1}. Moreover, \eqref{EqPropProjApprox3} and \eqref{EqPropProjApprox4} follow respectively from \eqref{EqPropProjApprox3a} and the construction. This finishes the proof.
 \end{proof}

Recalling the assumption that the sequence $\displaystyle\Exp_{n=1}^{N_i}\delta_{T^n (x)}$ converges to the measure $\rho$, one has that the sequence
$\displaystyle\Exp_{n=1}^{N_i}|f(T^nx)-h(T^nx)|$ converges to $\int_X|f-h|\di\rho$. Thus by the inequality \eqref{EqPropProjApprox1}, if $i$ is large enough then
\begin{equation}\label{EqBirkhoff}
\left|\Exp_{n=1}^{N_i}f(T^nx)\mu(n)-\Exp_{n=1}^{N_i}h(T^nx)\mu(n)\right|\le \Exp_{n=1}^{N_i}|f(T^nx)-h(T^nx)|<7\epsilon.
\end{equation}

Therefore, to prove \eqref{EqEpsilon}, it suffices to estimate $\left|\displaystyle\Exp_{n=1}^{N_i}h(T^nx)\mu(n)\right|$.

\subsection{Decomposition into short orbit segments}

Now let $C$ be a large constant for the moment, which will be specified later, and select a large integer $L$ with
\begin{equation}\label{EqL}C\frac{\log\log L}{\log L}<\epsilon.\end{equation}

Further define for each $1\leq j\leq L$ a subset $B_j\subset A_j$ by
\begin{equation}\label{EqB}B_j=\left\{y\in A_j:\Exp_{l=0}^{L-1}\bfone_{A_j}(T^ly)\geq1-\epsilon\right\},\end{equation}
which is clearly a compact subset.
Note that, by \eqref{EqPropProjApprox2},
\begin{align*}
1- \epsilon^2< \rho_j(A_j)=&\int_X \Exp_{l=0}^{L-1}1_{A_j}(T^l y)\di\rho_j(y)\\
\le&(1-\rho_j(A_j\setminus B_j))+\rho_j(A_j\setminus B_j)(1-\epsilon)\\
=&1-\rho_j (A_j\setminus B_j)\epsilon,
\end{align*}
thus $\rho_j (A_j\setminus B_j)<\epsilon$ and
$\rho_j(B_j)>1-2\epsilon$.
Similar to \eqref{EqTotalMass}, we deduce that
\begin{equation}\label{EqTotalMassB}\rho \left(\bigsqcup_{j=1}^JB_j\right)>(1-\epsilon)(1-2\epsilon)>1-3\epsilon.\end{equation}

Finally, let $\eta>0$ be small enough such that
\begin{equation} \label{EqEta}
|h(T^ly)-h(T^ly')|<\epsilon \text{ whenever }d(y,y')<\eta\ \text{and}\ 0\leq l\leq L-1.
\end{equation}

For every $n\in \mathbb{N}$, choose once and forever $x_n\in\bigsqcup_{j=1}^J B_j$ such that
\begin{equation*}\label{EqNearest}d(T^nx,x_n)=\min_{x'\in\bigsqcup_{j=1}^JB_j}d(T^nx,x').\end{equation*}

\begin{lemma}\label{LemRec}
$\displaystyle\Exp_{n=1}^{N_i}\bfone_{d(T^nx, x_n)\geq\eta}<3\epsilon$ once $i$ is large enough.
\end{lemma}

\begin{proof}
Applying the Tietze Extension Theorem we can find a continuous function $\phi: X\rightarrow [0,1]$ that equals $0$ on $\bigsqcup_{j=1}^JB_j$, and equals $1$ outside the $\eta$-open neighborhood of $\bigsqcup_{j=1}^JB_j$.
Then clearly \begin{equation}\label{EqLemRecPf1}\Exp_{n=1}^{N_i}\bfone_{d(T^nx, x_n)\geq \eta}\leq\Exp_{n=1}^{N_i}\phi(T^nx).\end{equation}

Since the sequence $\displaystyle\Exp_{n=1}^{N_i}\delta_{T^n (x)}$ converges to the measure $\rho$, which is supported on $\bigsqcup_{j=1}^JB_j$ except for a portion strictly smaller than $3\epsilon$ by \eqref{EqTotalMassB}, one has that
 the sequence $\displaystyle\Exp_{n=1}^{N_i}\phi(T^nx)$ converges to $\int_X\phi\di\rho<3\epsilon$ by the construction of the function $\phi$. And then the conclusion follows from \eqref{EqLemRecPf1}.
 \end{proof}

In the sequel, for each $i\in \mathbb{N}$ denote by $E_i$ the set of $n\in\{1,\cdots, N_i\}$ with $d(T^nx,x_n)<\eta$. Lemma \ref{LemRec} asserts $\# (E_i)>(1-3 \epsilon)N_i$ if $i$ is large enough.

\begin{corollary}\label{CorRec} For all sufficiently large $i$,
$$\left|\Exp_{n=1}^{N_i}h(T^nx)\mu(n)-\Exp_{n=1}^{N_i}\Exp_{l=0}^{L-1}h(T^lx_n)\mu(l+n)\right|<16\epsilon.$$\end{corollary}
\begin{proof}
Once $i$ is large enough, then $\frac L{N_i}<\frac{\epsilon}{2}$ and Lemma \ref{LemRec} holds.
As
\begin{equation*}
\begin{aligned}
&\Exp_{n=1}^{N_i}\Exp_{l=0}^{L-1}h(T^{l+n}x)\mu(l+n)= \sum_{n= 1}^{L- 1} \frac{n}{L N_i} h (T^n x) \mu (n) \\
&\hskip 66pt + \sum_{n= N_i+ 1}^{N_i+ L- 1} \frac{N_i+ L- n}{L N_i} h (T^n x) \mu (n)+ \frac{1}{N_i} \sum_{n=L}^{N_i}h(T^nx)\mu(n),
\end{aligned}
\end{equation*}
one has that, because of $|h|\leq 2$,
\begin{equation}\label{EqCorRecPf1}
\left|\Exp_{n=1}^{N_i}h(T^nx)\mu(n)-\Exp_{n=1}^{N_i}\Exp_{l=0}^{L-1}h(T^{l+n}x)\mu(l+n)\right|\leq \max_X|h|\cdot 2 \frac{L}{N_i}<2\epsilon.
\end{equation}
Moreover, by the construction \eqref{EqEta} of $\eta$ one has that $|h(T^{l+n}x)-h(T^lx_n)|$ $<\epsilon$ whenever $n\in E_i$ and $0\le l\le L-1$, and then
\begin{equation*}
\begin{aligned}
&\frac{1}{N_i} \left|\sum_{n\in E_i}\Exp_{l=0}^{L-1}h(T^{l+n}x)\mu(l+n)-\sum_{n\in E_i}\Exp_{l=0}^{L-1}h(T^lx_n)\mu(l+n)\right|\\
\le & \left|\Exp_{n\in E_i}\Exp_{l=0}^{L-1}h(T^{l+n}x)\mu(l+n)-\Exp_{n\in E_i}\Exp_{l=0}^{L-1}h(T^lx_n)\mu(l+n)\right|< \epsilon.
\end{aligned}
\end{equation*}
Observing that the density of the exceptional set $E_i^c$ in $\{1, \cdots, N_i\}$ is strictly smaller than $3\epsilon$ by Lemma \ref{LemRec}, we have
\begin{equation}\label{EqCorRecPf2}\begin{aligned}
&\left|\Exp_{n= 1}^{N_i}\Exp_{l=0}^{L-1}h(T^{l+n}x)\mu(l+n)-\Exp_{n= 1}^{N_i}\Exp_{l=0}^{L-1}h(T^lx_n)\mu(l+n)\right|\\
<&\epsilon+2\max_X|h|\cdot\frac{\# (E_i^c)}{N_i}<\epsilon+2\cdot2\cdot3\epsilon=13\epsilon.
\end{aligned}\end{equation}
Finally, the corollary is established by adding \eqref{EqCorRecPf2} to \eqref{EqCorRecPf1}.
\end{proof}

\subsection{Bounds on short averages}

At this point, it remains to control the short average $\displaystyle\Exp_{l=0}^{L-1}h(T^lx_n)\mu(l+n)$ for typical positions $n\in\{1,\cdots,N_i\}$. For this goal, we need the following number theoretical theorem of  Matom\"aki, Radziwi\l\l\ and Tao, taken from \cite{MRT15}.

\begin{proposition}\label{PropMRT}
There are constants $C_0$, $\kappa_0$ such that, for all $N\geq L\geq 10$,
$$\sup_{\beta\in\bT^1} \Exp_{n=0}^{N-1}\left|\Exp_{l=0}^{L-1}\mu(l+n)e(\beta l)\right|< C_0\left((\log N)^{-\kappa_0}+\frac{\log\log L}{\log L}\right).$$
\end{proposition}

The proposition is a direct application of \cite{MRT15}*{Theorem 1.7} to the M\"obius function $\mu$, following the discussion about the non-pretentiousness of $\mu$ preceding the theorem in that paper.

We now define the value of the constant $C$ to be
\begin{equation*}C=\left(\sum_{j=1}^J\sum_{m=1}^{M_j}|a_{j,m}|\right)\cdot C_0,\end{equation*}
where the coefficients $a_{j,m}$ are defined as in Proposition \ref{PropProjApprox}.

\begin{proposition}\label{PropShortAvg}
$\displaystyle\left|\Exp_{n=1}^{N_i}\Exp_{l=0}^{L-1}h(T^lx_n)\mu(l+n)\right|<6\epsilon$ once $i$ is large enough.
\end{proposition}
\begin{proof} For every $n\in \mathbb{N}$, the point $x_n$ belongs to $\bigsqcup_{j=1}^JB_j$ and thus lies in $B_{j_n}$ for some $1\leq j_n\leq J$.  Let $p_{j_n}$ and $h'_{j_n}$ be defined as in Proposition \ref{PropProjApprox}. Then by the construction \eqref{EqB} of the compact subset $B_{j_n}\subset A_{j_n}$, the set $F_n=\{0\leq l\leq L- 1: T^lx_n\in A_{j_n}\}$ has the cardinality $\# (F_n)\geq (1-\epsilon)L$.

 Thus for each $l\in F_n$, $x_n\in A_{j_n}$ and $T^l x_n\in A_{j_n}$, and then one has
\begin{equation*}\label{EqPropShortAvg1}h(T^lx_n)=h'_{j_n}(p_{j_n}(T^l x_n))=h'_{j_n}\big(R_{\alpha_{j_n}}^l(p_{j_n}(x_n))\big)\end{equation*}
by the properties \eqref{EqPropProjApprox3} and \eqref{EqPropProjApprox4}.
From this, we first deduce that
\begin{equation}\label{EqPropShortAvg2}\begin{aligned}
&\left|\Exp_{n=1}^{N_i}\Exp_{l=0}^{L-1}h(T^lx_n)\mu(l+n) -\Exp_{n=1}^{N_i}\Exp_{l=0}^{L-1}h'_{j_n}(R_{\alpha_{j_n}}^lp_{j_n}(x_n))\mu(l+n)\right|\\
= & \left|\Exp_{n=1}^{N_i}\frac1L\sum_{l\in\{0, 1,\cdots,L- 1\}\setminus F_n} [h (T^l x_n)- h'_{j_n}(R_{\alpha_{j_n}}^lp_{j_n} (x_n))]\mu(l+n)\right|\\
\leq & \Exp_{n=1}^{N_i} \frac1L\cdot\epsilon L\cdot (|h|+ |h'_{j_n}|)< 4\epsilon.
\end{aligned}\end{equation}

On the other hand, given \eqref{EqPropProjApprox4a}, for every $0\leq l\le L-1$, we can write
\begin{equation*}\begin{aligned}h'_{j_n}\big(R_{\alpha_{j_n}}^l(p_{j_n}(x_n))\big)
=&\sum_{m=1}^{M_{j_n}}a_{j_n,m}e\big(\xi_{j_n,m}\cdot R_{\alpha_{j_n}}^l\circ p_{j_n}(x_n)\big)\\
=&\sum_{m=1}^{M_{j_n}}a_{j_n,m}e\big(\xi_{j_n,m}\cdot p_{j_n}(x_n)+l(\xi_{j_n,m}\cdot \alpha_{j_n} )\big)\\
:=&\sum_{m=1}^{M_{j_n}}a_{j_n,m}e\big(\theta_{j_n,m}+l\beta_{j_n,m}\big).\end{aligned}\end{equation*}
Therefore,
\begin{equation*}\begin{aligned}
&\left|\Exp_{n=1}^{N_i}\Exp_{l=0}^{L-1}h'_{j_n}(R_{\alpha_{j_n}}^lp_{j_n}(x_n))\mu(l+n)\right|\\
=&\left|\Exp_{n=1}^{N_i}\Exp_{l=0}^{L-1}\sum_{m=1}^{M_{j_n}}a_{j_n,m}e\big(\theta_{j_n,m}+l\beta_{j_n,m}\big)\mu(l+n)\right|\\
\leq&\Exp_{n=1}^{N_i}\left|\sum_{m=1}^{M_{j_n}}a_{j_n,m}e\big(\theta_{j_n,m}\big) \Exp_{l=0}^{L-1}e\big(l\beta_{j_n,m}\big)\mu(l+n)\right|\\
\leq&\Exp_{n=1}^{N_i}\sum_{m=1}^{M_{j_n}}\left(|a_{j_n,m}|\cdot\left|\Exp_{l=0}^{L-1}e(l\beta_{j_n,m})\mu(l+n)\right|\right)\\
\leq&\Exp_{n=1}^{N_i}\sum_{j=1}^J\sum_{m=1}^{M_j}\left(|a_{j,m}|\cdot\left|\Exp_{l=0}^{L-1}e(l\beta_{j,m})\mu(l+n)\right|\right)\\
\leq&\left(\sum_{j=1}^J\sum_{m=1}^{M_j}|a_{j,m}|\right)\cdot
\sup_{\beta\in\bT^1}\Exp_{n=1}^{N_i}\cdot\left|\Exp_{l=0}^{L-1}e(l\beta)\mu(l+n)\right|.\end{aligned}\end{equation*}

By applying Proposition \ref{PropMRT} to the above estimate, we obtain that
\begin{equation*}\label{EqPropShortAvg4}\begin{aligned}
&\left|\Exp_{n=1}^{N_i}\Exp_{l=0}^{L-1}h'_{j_n}(R_{\alpha_{j_n}}^lp_{j_n}(x_n))\mu(l+n)\right|\\
<&\left(\sum_{j=1}^J\sum_{m=1}^{M_j}|a_{j,m}|\right)\cdot C_0\left((\log N_i)^{-\kappa_0}+\frac{\log\log L}{\log L}\right) \\
=&C\left((\log N_i)^{-\kappa_0}+\frac{\log\log L}{\log L}\right)
\end{aligned}\end{equation*}
once $i$ is large enough. Because that the second term is bounded by $\epsilon$ by \eqref{EqL}, one has that
\begin{equation}
\label{EqPropShortAvg3}
\begin{aligned}
\left|\Exp_{n=1}^{N_i}\Exp_{l=0}^{L-1}h'_{j_n}(R_{\alpha_{j_n}}^lp_{j_n}(x_n))\mu(l+n)\right|< 2\epsilon
\end{aligned}
\end{equation}
 as long as $C(\log N_i)^{-\kappa_0}<\epsilon$. Thus we can obtain the conclusion by adding together \eqref{EqPropShortAvg2} and \eqref{EqPropShortAvg3}.
\end{proof}

\begin{proof}[Proof of Theorem \ref{ThmGeneral}] Now we are ready to complete the proof of \eqref{EqEpsilon} (and hence the proof of Theorem \ref{ThmGeneral}), by adding together the estimates from the inequality \eqref{EqBirkhoff}, Corollary \ref{CorRec} and Proposition \ref{PropShortAvg}.\end{proof}


\begin{bibdiv}
\begin{biblist}


\bib{A59}{article}{
   author={Auslander, Joseph},
   title={Mean-$L$-stable systems},
   journal={Illinois J. Math.},
   volume={3},
   date={1959},
   pages={566--579},
   issn={0019-2082},
   review={},
}

\iffalse
\bib{B01}{article}{
   author={Berth{\'e}, Val{\'e}rie},
   title={Autour du syst\`eme de num\'eration d'Ostrowski},
   language={French, with French summary},
   journal={Bull. Belg. Math. Soc. Simon Stevin},
   volume={8},
   date={2001},
   number={2},
   pages={209--239},
}
\fi

\bib{B13b}{article}{
   author={Bourgain, J.},
   title={M\"obius-Walsh correlation bounds and an estimate of Mauduit and
   Rivat},
   journal={J. Anal. Math.},
   volume={119},
   date={2013},
   pages={147--163},
}

\bib{B13a}{article}{
   author={Bourgain, J.},
   title={On the correlation of the Moebius function with rank-one systems},
   journal={J. Anal. Math.},
   volume={120},
   date={2013},
   pages={105--130},
}


\bib{BSZ13}{article}{
   author={Bourgain, J.},
   author={Sarnak, P.},
   author={Ziegler, T.},
   title={Disjointness of M\"obius from horocycle flows},
   conference={
      title={From Fourier analysis and number theory to Radon transforms and
      geometry},
   },
   bool={
      series={Dev. Math.},
      volume={28},
      publisher={Springer, New York},
   },
   date={2013},
   pages={67--83},
}

\bib{D37}{article}{
   author={Davenport, H.},
   title={On some infinite series involving arithmetical functions II},
   journal={Quat. J. Math.},
   volume={8},
   date={1937},
   pages={313--320},
}



\bib{DG15}{article}{
   author={Downarowicz, Tomasz},
   author={Glasner, Eli},
   title={Isomorphic extensions and applications},
   journal={Topological Methods in Nonlinear Analysis},
   volume={},
   date={2015},
   number={},
   pages={to appear, arXiv:1502.06999v1},
}



\bib{DK15}{article}{
   author={Downarowicz, Tomasz},
   author={Kasjan, Stanis{\l}aw},
   title={Odometers and Toeplitz systems revisited in the context of
   Sarnak's conjecture},
   journal={Studia Math.},
   volume={229},
   date={2015},
   number={1},
   pages={45--72},
   issn={0039-3223},
   review={},
}

\bib{EKSL16}{article}{
   author={El Abdalaoui, El Houcein},
   author={Kasjan, Stanis{\l}aw},
   author={Lema{\'n}czyk, Mariusz},
   title={0-1 sequences of the Thue-Morse type and Sarnak's conjecture},
   journal={Proc. Amer. Math. Soc.},
   volume={144},
   date={2016},
   number={1},
   pages={161--176},
   issn={0002-9939},
   doi={},
}

\bib{EKLR15}{article}{
   author={El Abdalaoui, El Houcein},
   author={Ku{\l}aga-Przymus, J.},
   author={Lema{\'n}czyk, Mariusz},
   author={de la Rue, Thierry},
   title={The Chowla and the Sarnak conjectures from ergodic theory point of view},
   journal={Discrete Contin. Dyn. Syst.},
   date={2016},
   number={},
   pages={to appear, arXiv:1410.1673v3},
}

\bib{ELD14}{article}{
   author={El Abdalaoui, El Houcein},
   author={Lema{\'n}czyk, Mariusz},
   author={de la Rue, Thierry},
   title={On spectral disjointness of powers for rank-one transformations
   and M\"obius orthogonality},
   journal={J. Funct. Anal.},
   volume={266},
   date={2014},
   number={1},
   pages={284--317},
}

\bib{ELD15}{article}{
   author={El Abdalaoui, El Houcein},
   author={Lema{\'n}czyk, Mariusz},
   author={de la Rue, Thierry},
   title={Automorphisms with quasi-discrete spectrum, multiplicative functions and average orthogonality along short intervals},
   journal={International Mathematics Research Notices},
   volume={},
   date={2016},
   number={},
   pages={to appear},
}


\bib{FJ15}{article}{
   author={Fan, Aihua},
   author={Jiang, Yunping},
   title={Oscillating sequences, minimal mean attractability and minimal mean-Lyapunov-stability},
   journal={preprint},
   volume={},
   date={2015},
   number={},
   pages={arXiv:1511.05022v1},
}



\bib{FKLM15}{article}{
   author={Ferenzi, S.},
   author={Ku\l{}aga-Przymus, Joanna},
   author={Lema\'nczyk, Mariusz},
   author={Mauduit, C.},
   title={Substitutions and M\"obius disjointness},
   journal={preprint},
   date={2015},
   pages={arXiv:1507.01123v1},
}


\bib{F51}{article}{
   author={Fomin, S.},
   title={On dynamical systems with a purely point spectrum},
   language={Russian},
   journal={Doklady Akad. Nauk SSSR (N.S.)},
   volume={77},
   date={1951},
   pages={29--32},
   review={},
}


\iffalse
\bib{F61}{article}{
   author={Furstenberg, H.},
   title={Strict ergodicity and transformation of the torus},
   journal={Amer. J. Math.},
   volume={83},
   date={1961},
   pages={573--601},
}

\bib{F67}{article}{
   author={Furstenberg, Harry},
   title={Disjointness in ergodic theory, minimal sets, and a problem in
   Diophantine approximation},
   journal={Math. Systems Theory},
   volume={1},
   date={1967},
   pages={1--49},
   issn={0025-5661},
   review={},
}


\bib{F81}{book}{
   author={Furstenberg, H.},
   title={Recurrence in ergodic theory and combinatorial number theory},
   note={M. B. Porter Lectures},
   publisher={Princeton University Press, Princeton, N.J.},
   date={1981},
   pages={xi+203},
   isbn={0-691-08269-3},
   review={},
}
\fi


\bib{G12}{article}{
   author={Green, Ben},
   title={On (not) computing the M\"obius function using bounded depth
   circuits},
   journal={Combin. Probab. Comput.},
   volume={21},
   date={2012},
   number={6},
   pages={942--951},
}

\bib{GT12}{article}{
   author={Green, Ben},
   author={Tao, Terence},
   title={The M\"obius function is strongly orthogonal to nilsequences},
   journal={Ann. of Math. (2)},
   volume={175},
   date={2012},
   number={2},
   pages={541--566},
}

\bib{HVN42}{article}{
   author={Halmos, Paul R.},
   author={von Neumann, John},
   title={Operator methods in classical mechanics. II},
   journal={Ann. of Math. (2)},
   volume={43},
   date={1942},
   pages={332--350},
   issn={0003-486X},
}

\bib{HLSY15}{article}{
   author={Huang, Wen},
   author={Lian, Zhengxing},
   author={Shao, Song},
   author={Ye, Xiangdong},
   title={Sequences from zero entropy noncommutative toral automorphisms and Sarnak Conjecture},
   journal={preprint},
   volume={},
   date={2015},
   number={},
   pages={arXiv:1510.06022v1},
}


\bib{IK04}{book}{
   author={Iwaniec, Henryk},
   author={Kowalski, Emmanuel},
   title={Analytic number theory},
   series={American Mathematical Society Colloquium Publications},
   volume={53},
   publisher={American Mathematical Society, Providence, RI},
   date={2004},
   pages={xii+615},
   isbn={0-8218-3633-1},
   review={},
   doi={10.1090/coll/053},
}


\bib{KL15}{article}{
   author={Ku{\l}aga-Przymus, J.},
   author={Lema{\'n}czyk, M.},
   title={The M\"obius function and continuous extensions of rotations},
   journal={Monatsh. Math.},
   volume={178},
   date={2015},
   number={4},
   pages={553--582},
}

\iffalse
\bib{K97}{book}{
   author={Khinchin, A. Ya.},
   title={Continued fractions},
   edition={Translated from the third (1961) Russian edition},
   note={With a preface by B. V. Gnedenko;
   Reprint of the 1964 translation},
   publisher={Dover Publications, Inc., Mineola, NY},
   date={1997},
   pages={xii+95},
}
\fi


\bib{L87}{article}{
   author={Lehrer, Ehud},
   title={Topological mixing and uniquely ergodic systems},
   journal={Israel J. Math.},
   volume={57},
   date={1987},
   number={2},
   pages={239--255},
   issn={0021-2172},
   review={\MR{890422}},
   doi={10.1007/BF02772176},
}


\bib{LTY15}{article}{
   author={Li, Jian},
   author={Tu, Siming},
   author={Ye, Xiangdong},
   title={Mean equicontinuity and mean sensitivity},
   journal={Ergodic Theory Dynam. Systems},
   volume={35},
   date={2015},
   number={8},
   pages={2587--2612},
   issn={0143-3857},
   review={},
}

\bib{LS15}{article}{
   author={Liu, Jianya},
   author={Sarnak, Peter},
   title={The M\"obius function and distal flows},
   journal={Duke Math. J.},
   volume={164},
   date={2015},
   number={7},
   pages={1353--1399},
   issn={0012-7094},
}

\bib{MMR14}{article}{
   author={Martin, Bruno},
   author={Mauduit, Christian},
   author={Rivat, Jo{\"e}l},
   title={Th\'eor\'eme des nombres premiers pour les fonctions digitales},
   language={French},
   journal={Acta Arith.},
   volume={165},
   date={2014},
   number={1},
   pages={11--45},
}

\bib{MRT15}{article}{
   author={Matom{\"a}ki, Kaisa},
   author={Radziwi{\l}l, Maksym},
   author={Tao, Terence},
   title={An averaged form of Chowla's conjecture},
   journal={Algebra Number Theory},
   volume={9},
   date={2015},
   number={9},
   pages={2167--2196},
   issn={1937-0652},
   review={},
   doi={},
}

\bib{MR10}{article}{
   author={Mauduit, Christian},
   author={Rivat, Jo{\"e}l},
   title={Sur un probl\`eme de Gelfond: la somme des chiffres des nombres
   premiers},
   language={French, with English and French summaries},
   journal={Ann. of Math. (2)},
   volume={171},
   date={2010},
   number={3},
   pages={1591--1646},
}

\bib{MR15}{article}{
   author={Mauduit, Christian},
   author={Rivat, Jo{\"e}l},
   title={Prime numbers along Rudin-Shapiro sequences},
   journal={J. Eur. Math. Soc. (JEMS)},
   volume={17},
   date={2015},
   number={10},
   pages={2595--2642},
   issn={1435-9855},
}

\bib{M77}{book}{
   author={Morris, Sidney A.},
   title={Pontryagin duality and the structure of locally compact abelian
   groups},
   note={London Mathematical Society Lecture Note Series, No. 29},
   publisher={Cambridge University Press, Cambridge-New York-Melbourne},
   date={1977},
   pages={viii+128},
   review={},
}


\iffalse
\bib{O22}{article}{
   author={Ostrowski, Alexander},
   title={Bemerkungen zur Theorie der Diophantischen Approximationen},
   language={German},
   journal={Abh. Math. Sem. Univ. Hamburg},
   volume={1},
   date={1922},
   number={1},
   pages={77--98},
}
\fi


\bib{O52}{article}{
   author={Oxtoby, John C.},
   title={Ergodic sets},
   journal={Bull. Amer. Math. Soc.},
   volume={58},
   date={1952},
   pages={116--136},
   issn={0002-9904},
   review={},
}

\bib{P15}{article}{
   author={Peckner, Ryan},
   title={M\"obius disjointness for homogeneous dynamics},
   journal={preprint},
   date={2015},
   pages={arXiv:1506.07778v1},
}

\bib{S09}{article}{
   author={Sarnak, Peter},
   title={Three lectures on the M\"obius function, randomness and dynamics},
   journal={lecture notes, IAS},
   date={2009},
}

\bib{S12}{article}{
   author={Sarnak, Peter},
   title={Mobius randomness and dynamics},
   journal={Not. S. Afr. Math. Soc.},
   volume={43},
   date={2012},
   number={2},
   pages={89--97},
}


\bib{Veech16}{article}{
   author={Veech, William A.},
   title={Moebius orthogonality for generalized Morse-Kakutani flows},
   journal={Amer. J. Math.},
   volume={},
   date={2016},
   number={},
   pages={to appear},
   issn={},
}


\iffalse
\bib{V77}{article}{
   author={Vaughan, Robert-C.},
   title={Sommes trigonom\'etriques sur les nombres premiers},
   language={French, with English summary},
   journal={C. R. Acad. Sci. Paris S\'er. A-B},
   volume={285},
   date={1977},
   number={16},
}
\fi


\bib{W82}{book}{
   author={Walters, Peter},
   title={An introduction to ergodic theory},
   series={Graduate Texts in Mathematics},
   volume={79},
   publisher={Springer-Verlag, New York-Berlin},
   date={1982},
   pages={ix+250},
   isbn={0-387-90599-5},
   review={},
}

\bib{W15}{article}{
   author={Wang, Zhiren},
   title={M\"obius disjointness for analytic skew products},
   journal={preprint},
   date={2015},
   pages={arXiv:1509.03183v2},
}


\bib{W16}{book}{
   author={Wei, Fei},
   title={Entropy of arithmetic functions and Sarnak's M\"obius disjointness conjecture},
   note={Thesis (Ph.D.)--The University of Chinese Academy of Sciences},
   publisher={},
   date={2016},
   pages={},
   isbn={},
   review={},
}

\end{biblist}
\end{bibdiv}
\end{document}